\documentclass[10pt]{article}

\pdfoutput=1
\usepackage[pdftex]{color}
\usepackage{amssymb}
\usepackage{amsthm} 
\usepackage{amsmath}
\usepackage{latexsym}
\usepackage{amscd}
\usepackage{graphicx}
\usepackage[pdftex, colorlinks=true, citecolor=green]{hyperref}
\usepackage{lscape}
\usepackage{multirow}
\usepackage{wrapfig}

\newcommand{\ds}{\displaystyle}

\newcommand{\ben}{\begin{equation}}     
\newcommand{\eeqn}{\end{equation}}
\newcommand{\bey}{\begin{eqnarray}}
\newcommand{\eey}{\end{eqnarray}}

\newtheorem{thm}{Theorem}[section]

\newtheorem{lemma}[thm]{Lemma}
\newtheorem{corol}[thm]{Corollary}

\newtheorem{conj}[thm]{Conjecture}

\begin{document}


\vspace{5mm}
\noindent {\Large
\textbf{Effects of local mutations in quadratic iterations}
}
\\\\
\indent \emph{Anca R\v{a}dulescu$^{*,}\footnote{Associate Professor, Department of Mathematics, State University of New York at New Paltz; New York, USA; Phone: (845) 257-3532; Email: radulesa@newpaltz.edu}$, Abraham Longbotham$^{2}$, Ashelee Collier$^1$}

\vspace{2mm}
\indent $^1$ Department of Mathematics, SUNY New Paltz
\\
\indent $^2$ Department of Physics,  SUNY New Paltz


\begin{abstract}
We introduce mutations in replication systems in which the intact copying mechanism is performed by discrete iterations of a complex quadratic map in the family $f_c(z) = z^2+c$. More specifically, we consider a ``correct" function $f_{c_1}$ acting on the complex plane (representing the RNA to be copied). A ``mutation"  $f_{c_0}$ is a different (``erroneous'') map acting on a locus of given radius $r$ around a mutation focal point $\xi^*$. The effect of the mutation is interpolated radially to eventually recover the original map $f_{c_1}$ when reaching an outer radius $R$. We call the resulting map a ``mutated'' map. 

In the theoretical framework of mutated iterations, we study how a mutation (replication error) affects the temporal evolution of the system, in the context of cellular differentiation. We use the prisoner set of the system to quantify simultaneously the long-term behavior of the entire space under mutated maps. We analyze how the position, timing and size of the mutation can alter the system's long-term evolution (as encoded in the topology of its prisoner set). In the context of genetics, this framework may increase our understanding of the factors and mechanisms that shape the genetic expression, in a specialized cell, in the process of differentiation from a stem cell.
\end{abstract}

\vspace{5mm}
\noindent {\bf Keywords:} replication error, perturbed iterations, prisoner set, cell differentiation.

\section{Introduction}

\subsection{Modeling cell differentiation}

All cells in any living multicellular organism originate from stem cells~\cite{frank2007dynamics}. These are pluripotent cells,  which can diversify into a wide array of cell types, with dramatic structural and functional differences. To achieve this variety of cellular profiles, required for the function of the organism as a whole, stem cells undergo a process of differentiation, which determines their future morphology, physical properties (such as cross membrane potential and responsiveness to electrical signals), metabolic activity, etc. Empirical studies have shown that cell differentiation generally depends on highly controlled changes in gene expression, rather than on any changes in the nucleotide sequence of the cell's genome~\cite{gilbert2010developmental,alberts2017molecular}. This differentiation path is modulated by a variety of factors, including environmental influences, signaling from other cells and developmental stage~\cite{brun2020fit}. A variety of mathematical contexts and methods have been used to date to model the process of cell differentiation~\cite{pazdziorek2014mathematical,prokharau2014mathematical,li2023mathematical}, including limited physiological detail, so that analytical tractability is preserved.

However, in the process of differentiation, RNA replication may not be perfect. Empirical studies  found that there are a variety of mutations that can affect the differentiation process in significant and impactful ways. For example, a study of induced pluripotent stem cells (iPSCs) found that such marked differences in gene expression between incorrectly and successfully differentiated lines contribute to inhibition of neurogenesis in the former (with the results obtained in culture likely reflecting similar associations in actual brain development). The factors and mechanisms behind these significant variations remain largely unknown, even with the current state of the art technoligies~\cite{puigdevall2023somatic}. This makes it very difficult to assemble a mathematical model of mutation that considers all factors impacting gene expression that are valuable in a physiological context. Existing models focus on one factor at a time, and typically on very specific cell types and circumstances~\cite{ryman2008effect}. However, some of the important questions that remain unanswered are very general, such as comparing the degree of divergence for different mutation rates, or comparing effects of mutation in the early versus later phases of the divergence~\cite{ryman2008effect}. This suggests that a general theoretical framework could be useful for contextualizing these questions in a canonical model, in which the specific details have been removed, retaining only the phenomenology of a DNA replication mechanism subject to mutations.

\subsection{Modeling intact replication}

In our work, we use a complex dynamics model to help us better represent theoretically the principles underlying perfect and imperfect cell differentiation, and to analyze the difference between them. We use the complex plane $\mathbb{C}$ to represent symbolically a cell's RNA, with each point in the plane corresponding to the expression of a single gene. We define a perfect replication system as the one-dimensional, discrete time dynamical system generated by iterations of the \emph{intact} transformation$f \colon \mathbb{C} \to \mathbb{C}, \; f_{c_1} = z^2+c_1$, with $c_1 \in \mathbb{C}$, applied on the whole complex plane. Our modeling premise is that the initial points $z_0$ that remain asymptotically bounded under iterations of the replicator $f_{c_1}$ represent those ``genes'' that are expressed in the terminally differentiated cell (i.e., the cell resulting at the end of the differentiation process). 

The locus of the points with bounded orbits under $f_{c_1}$ is widely known in the complex dynamics literature as the prisoner set ${\cal P}(f_{c_1})$. Hence, in our model, the prisoner set ${\cal P}(f_{c_1})$ is a symbolic representation of the gene expression in the perfectly differentiated cell. Along these lines, one can think of different parameters $c_1$ as different built-in replicators, which encompass all internal and external factors that determine the type of cell to be obtained. Different cell types will therefore have different terminal genetic expression, represented by corresponding prisoner sets ${\cal P}(f_{c_1})$ with different topologies (see Figure~\ref{intact_prisoner}).

The traditional theory of complex dynamics in the quadratic family provides an escape radius for the iterative process, that is: if the orbit of a point $z_0$ exits the disc of radius two at any iteration step, it will never return to the disc, and will escape to infinity (i.e., $z_0$ is not in the prisoner set). This property can be capitalized as a mechanism of tracking the replication step where a gene is discarded as a potential contender for the genetic locus expressed in the terminally differentiated cell. This is an interesting path to investigate, since the temporal evolution from stem cells to their differentiated children is characterized by complex changes in cellular morphology and function.

Rephrasing theoretical questions on RNA replication and cell differentiation in terms of complex map iterations presents a mathematically approachable and canonical way that may lead to in interesting results (both in terms of
mathematical novelty, and improving our understanding of mechanisms behind gene expression). Complex plane iterations represent good modeling candidates: not only is their dynamics most thoroughly studied and well understood, but it can also be illustrated geometrically by sets with extraordinarily
complex, but topologically quantifiable structure. This structure can be used to understand, describe and
classify the long term dynamics of the system, and are perhaps not far from the actual physiological complexity of cellular structure and genetics~\cite{albrecht2012fractal,dokukin2015emergence,moreno2011human}.

\begin{figure}[h!]
\begin{center}
\includegraphics[width=0.7\textwidth]{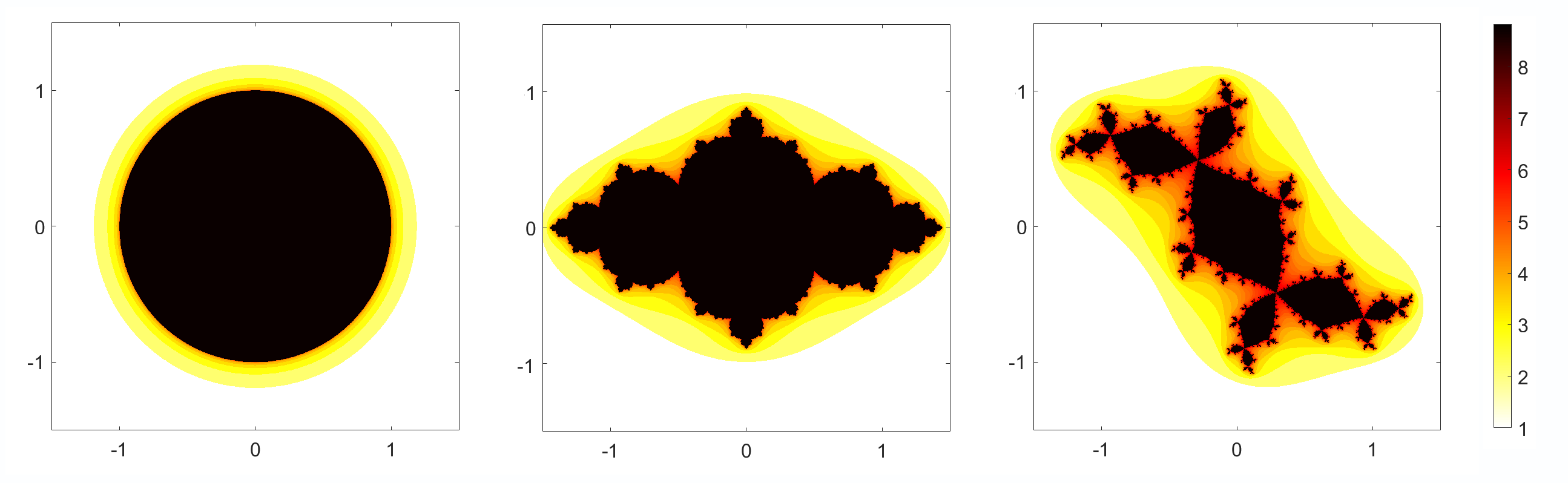}
\end{center}
\caption{\small \emph{{\bf Prisoner sets corresponding to intact mutations} for the generating parameters $c_1=0$ (unit disc), $c_1=-0.65$ (basilica) and $c_1=-0.13+0.77i$ (Douady rabbit). These can be interpreted as being the gene expression loci for three different types of cells, with different morphology and function (e.g., epithelial cell, kidney cell and neuron). In each panel, the prisoner set is shown in black (the points that never escape). The other colors represent the number of steps in which each point in $\mathbb{C}$ escapes the disc or radius $R_e=2$, under iterations of the corresponding map $f_{c_1}$ (in our interpretation, this represents the replication step where the gene stops being carried forward in the expression locus). The same color map will be used in all our other illustrations of prisoner sets, which are all computed in $800 \times 800$ resolution, and based on 100 total iterations.}}
\label{intact_prisoner}
\end{figure}

\subsection{Modeling mutation}

In this context, a mutation can be seen as an error which causes the replication mechanism $f_{c_1}$ to act as a perturbed replicator $f_{c_0}$ at some or all of the replication steps. Since mutations do not typically affect the whole cellular RNA, but rather a local portion of it (often an isolated gene), it makes sense to define a mutation as a local perturbation of the correct function $f_{c_1}$, such that the erroneous replication $f_{c_0}$ acts only on a small region of the complex plane (region which will be allowed to vary in size from one focal point to a considerably larger locus). For this study, we assume this mutation to be persistent (i.e., to act upon the same region in the plane at each iteration step). Within this setup, we investigate to what extent a local persistent mutation can affect the final genetic expression in the terminal differentiation of the cell, by monitoring the shape and properties of the prisoner set. In future work we will expand our model to include more general mutations that can act only temporarily, or can switch in time between target loci.

\color{black}

More precisely, a mutation $f_{c_0} = z^2+c_0$ is set to act on a disc of radius $r \geq 0$ around the focal point $\xi^*$, radiating a broader transient effect on a larger disc of radius $R \geq r$, but with no further effect to the correct function outside of this disc (i.e., points that fall outside the disc or radius $R$ at one time step will be simply iterated under the intact map $f_{c_1}$ at the next time step). The function on the annulus represents a fading effect of the mutation $c_0$ applied inside the smaller disc, transitioning towards the correct transformation $c_1$ defined outside the larger disc. To accomplish this effect, we chose the following generating function for our mutated iterations:

\begin{equation}
f(z) = \left\{  \begin{array}{ll} f_{c_0}(z), & \lvert z - \xi^*\rvert \leq r \\  f_{01}(z), & r < \lvert z -\xi^* \rvert < R\\ f_{c_1}(z), & \lvert z-\xi^* \rvert \geq R \end{array} \right.
\label{function}
\end{equation}

\noindent  defined with radial symmetry. More precisely, for $0\leq t \leq1$, and $\theta \in [0,2\pi]$, then we define the transition function on the annulus as
$$f_{01}(z(t, \theta)) = [z(t, \theta)]^2 + c_0(1-t) + c_1t$$

\noindent where $\rho(t) = r(1-t) + Rt$ and $z(t,\theta) = \xi^*+ \rho(t)[\cos(\theta) + i \sin(\theta)]$. For simplicity, in the rest of the paper we will restrict ourselves to the case $\xi^* =0$, so that the mutation is centered at the origin (see Figure~\ref{basic_mutation}). Under this assumption,  for $\rho = \lvert z \rvert = r(1-t) + Rt$ (with $0\leq t\leq 1$), we can rewrite:
$$f_{01}(z)=z^2 + c_0\frac{R-\rho}{R-r} + c_1\frac{\rho -r}{R-r}$$

\begin{figure}[h!]
\begin{center}
\includegraphics[width=0.7\textwidth]{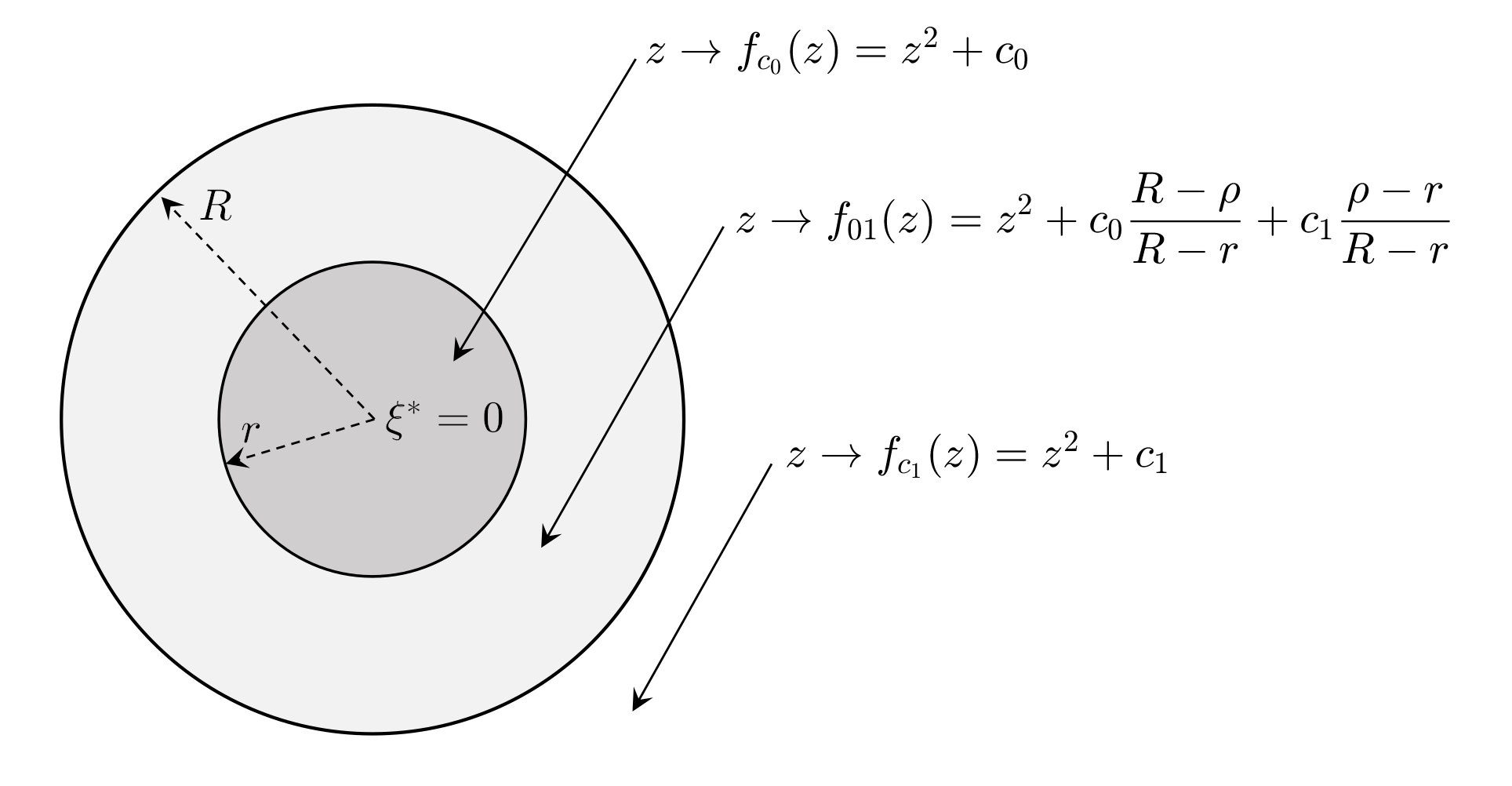}
\end{center}
\caption{\small \emph{{\bf Mutation applied at the origin,} illustrating the mutation focus, mutation disk and transition annulus, and showing the corresponding formulas for the iterated function $z \to f(z)$ on each portion of the complex plane.}}
\label{basic_mutation}
\end{figure}

\noindent Clearly, one expects to see differences between the prisoner sets of the intact function ${\cal P}(f_{c_1})$ and that of the mutated system ${\cal P}(f)$, with genes which are in the expression locus of the intact terminal differentiation, but not in the locus of the mutated differentiation, and conversely. Our modeling framework allows us to study how the differences between the two expression loci depend on the parameters $\xi^*$, $r$, $R$, $c_1$ and $c_0$. Since, as previously discussed, the actual physiological factors that underlie mutations are largely unknown, this represents an efficient theoretical framework to investigate, interpret and compare the consequences of changing the position, size or extent of mutation.\\

\color{black}
\noindent The rest of the manuscript is organized as follows. In Section~\ref{methods}, we describe the basic properties of the model considered in this paper. In Section~\ref{results}, we present analytical results and numerical simulations that illustrate the effects of local mutations on the replication outcome. In Section~\ref{discussion}, we interpret the results, discuss their relevance to understanding the role of mutations in genetic expression, and describe potential extensions and ideas for future work.

\section{Analysis of mutated systems}
\label{methods}

\begin{lemma} The function constructed in~\eqref{function} is continuous.
\end{lemma}

\proof{The function is trivially continuous within each of the three regions $| z | < r$, $r < |z| < R$ and $|z| >R$. We yet have to prove that the function is continuous on the boundaries between these regions. We will prove continuity at an arbitrary point $z=z_0$ on the smaller circle of radius $r$.

Suppose $| z_0 | =r$, and fix $\varepsilon>0$. We want to find $\delta>0$ such that, for all $| z | \geq r$ with $|z-z_0| < \delta$, we have $|f(z)-f(z_0)| <\varepsilon$. For any $| z | \geq r$, we calculate.
\begin{equation}
|f(z)-f(z_0)| = \left \lvert z^2 + c_0\frac{R-\rho}{R-r} + c_1\frac{\rho -r}{R-r} - z_0^2 -c_0 \right \rvert \leq |z^2-z_0^2| + |c_1-c_0| \cdot \frac{\rho-r}{R-r} \nonumber
\end{equation}

\noindent Notice that $\rho-r =|z|-|z_0| \leq |z-z_0| < \delta$ and that $|z+z_0| \leq |z-z_0| + 2|z_0| < \delta +2r$. Hence:
$$|f(z)-f(z_0)| \leq \delta(\delta+2r) + \frac{|c_1-c_0|\delta}{R-r} = \delta \left[ \delta +2r +\frac{\lvert c_1-c_0 \rvert}{R-r} \right]$$

\noindent It follows immediately that, for $\delta$ sufficiently small, $|f(z)-f(z_0)| <\varepsilon$. A similar reasoning can apply to show continuity at an arbitrary point on the larger circle of radius $R$, hence the proof will be skipped.

\hfill \qed}\\

\begin{lemma} The function constructed in~\eqref{function} is not analytic for $c_0 \neq c_1$.
\end{lemma}

\proof{We use the Cauchy-Riemann conditions in polar coordinates to determine whether $f$ is analytic within the annulus of radius $r < \rho < R$, where the function is defined as: 
$$f(z) = f(\rho, \theta) = \rho^2[\cos (2\theta) + i\sin (2\theta)] + c_0\frac {R-\rho}{R-r} + c_1\frac {\rho-r}{R-r}$$

\noindent If we write the complex parameters as $c_0=a_0+ib_0$, and $c_1=a_1+ib_1$, then $\ds u = \text{Re}(f) = \rho^2\cos2\theta + a_0\frac {R-\rho}{R-r} + a_1\frac {\rho-r}{R-r}$ and $\ds v = \text{Im}(f) = \rho^2\sin2\theta + b_0\frac {R-\rho}{R-r} + b_1\frac {\rho-r}{R-r}$. We compute\\

$\ds \frac{\partial u}{\partial \rho} = 2\rho\cos(2\theta) + \frac {a_1-a_0}{R-r} \; \text{ and } \; \frac{\partial u}{\partial \theta} = -2\rho^2 \sin(2\theta)$\\

$\ds \frac{\partial v}{\partial \rho} = 2\rho \sin(2\theta) +\frac{b_1-b_0}{R-r} \; \text{ and } \; \frac{\partial v}{\partial \theta} =2\rho^2 \cos(2\theta)$\\

\noindent The Cauchy-Riemann conditions are equivalent in this care with:

$$\frac{\partial u}{\partial \rho} = \frac {1}{\rho}\frac {\partial v}{\partial \theta} \; \Longleftrightarrow \; \frac{a_1-a_0}{R-r} = 0$$ 

$$\frac {\partial v}{\partial \rho} = -\frac {1}{\rho}\frac {\partial u}{\partial \theta} \; \Longleftrightarrow \; \frac{b_1-b_0}{R-r} = 0$$ 

\noindent Hence, the function is only analytic in the annulus for $c_0=c_1$, and approaches analyticity if $R \to \infty$.

\hfill \qed}

\vspace{3mm}
\noindent One of the crucial basic properties of intact single quadratic map iterations is the existence of an escape radius. Below, we will show that this property also holds for mutated iterations, and that $M = 1+\sqrt{1+c}$, where $c= \lvert c_0 \rvert + |c_1|$, acts as an escape radius for the mutated iteration of $f$, defined in~\eqref{function}. 

First, suppose that $r<M<R$. Then we have the following lemma, reflecting expansion within the transition annulus:

\begin{lemma} If $|z| \geq M$ then $|f_{01}(z)| \geq 2|z|$.\\
\end{lemma}

\begin{proof}

\begin{eqnarray*}
|f_{01}(z)| - 2|z| &=& \left |  z^2 + c_0\frac{R-\rho}{R-r} + c_1\frac{\rho -r}{R-r} \right| -2|z| \geq |z|^2 - |c_0|\left|\frac{R-\rho}{R-r}\right| - |c_1|\left|\frac{\rho -r}{R-r}\right|-2|z|\\ \\
	          & \geq& |z|^2 - (|c_0| + |c_1|) -2|z|  = |z|^2 - c - 2|z|  = (|z| - 1)^2 - (1 + c)\\\\
		   &=& \left( |z| - 1 - \sqrt{1 + c} \right) \left( |z| - 1 + \sqrt{1 + c} \right) = \left( |z| - 1 + \sqrt{1 + c} \right)\left( |z| - M \right)
\end{eqnarray*}

\noindent Suppose $|z| \geq M = 1 + \sqrt{1 + c}$. It follows that $|z| \geq 1 - \sqrt{1 + c}$, hence $|f_{01}(z)| - 2|z| \geq 0$, so that $|f_{01}(z)| \geq 2|z|$, which is what we needed to prove.

\end{proof}

\noindent It is easy to show that the expansion condition applies to the functions $f_{c_0}$ and $f_{c_1}$ defined outside the annulus, from the following lemma:

\begin{lemma} If $|z| \geq M$, then $|f_{c_i}(z)| \geq 2|z|$, for $i \in \{ 0,1 \}$.
\end{lemma}

\begin{proof}
For $i=0,1$ we have:
\begin{eqnarray*}
|f_{c_i}(z)| -2|z|&=& |z^2+c_i| -2|z| \geq |z|^2 - |c_i| -2|z| >|z|^2-2|z| -c\\
&=& (|z|-1)^2 - (c+1) = (|z|-M)(|z|-1+\sqrt{c+1}) >0 \text{ for } |z| \geq M
\end{eqnarray*}

\noindent Hence $|f_{c_i}(z)| \geq 2|z|$, for $|z| \geq M$ and $i=0,1$.
\end{proof}

\noindent Taken together, the two lemmas lead to the following:
\begin{thm} The positive constant $M=1+\sqrt{1+ |c_0|+|c_1|}$ is an escape radius for the iteration defined by~\eqref{function}. More precisely, if $|z| \geq M$, then $|f(z)| \geq 2|z|$; thus $z_n =  f^{\circ n}(z_0)$ escapes for any $|z_0| \geq M$.
\end{thm}

\begin{figure}[h!]
\begin{center}
\includegraphics[width=0.7\textwidth]{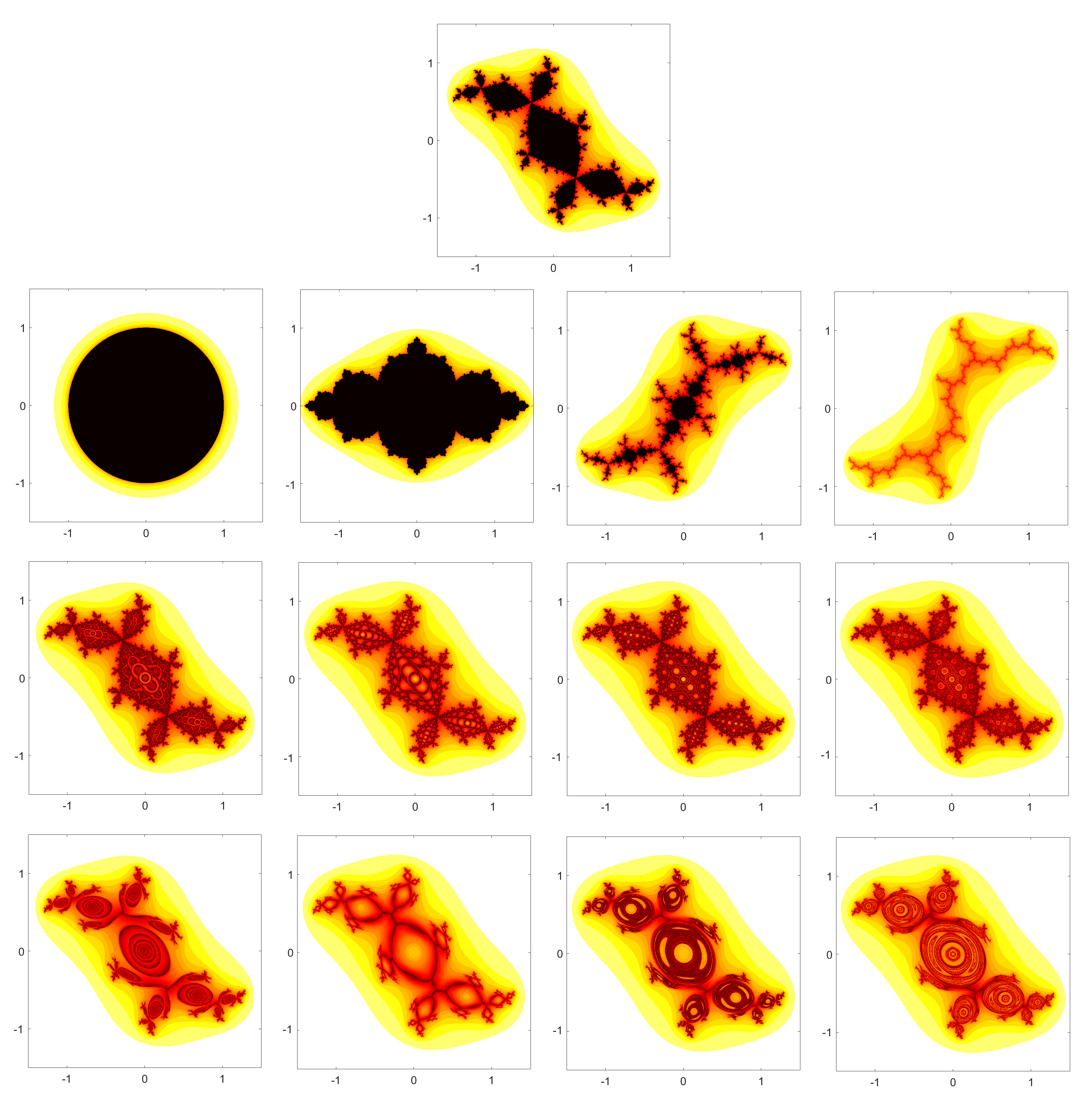}
\end{center}
\caption{\small \emph{{\bf Example of prisoner sets for different mutations.} The top panel illustrates the prisoner set of the intact map $f_{c_1}$, for $c_1=-0.13+0.77i$ (the Douady rabbit). The next row represents the prisoner sets for four different functions $f_{c_0}$, from left to right: $c_0=0$ (unit disk); $c_0=-0.65$ (basilica); $c_0=-0.117-0.856i$; $c_0=-i$ (dendrite). The two additional panels in each column illustrate the prisoner set for a system with a point-wise mutation $c_0$ at the origin, and transition radius $R=0.1$ (third row) and $R=0.5$ (fourth row). }}
\label{parameter_cases}
\end{figure}

\section{Dependence on mutation radii}
\label{results}

\noindent In the next sections, we will investigate numerically and analytically the effects of various key parameters on the asymptotic dynamics of the replication system, as reflected by the geometric properties of its prisoner set. For example, Figure~\ref{parameter_cases} illustrates the prisoner sets for four systems in which different point-wise mutations (with $r=0$) were inserted at the origin into the same intact replication system $c_1=-0.13+0.77i$ (the Douady rabbit), for two different values of the transition radius $R$. The figure suggests that the shape of the prisoner set varies widely based not only on the parameter pair $(c_0,c_1)$, but also on the size of the two radii $0<r \leq R$. It is easy to see that ${\cal P}(f)$ approaches ${\cal P}(f_{c_1})$ when $R \searrow 0$, and approaches ${\cal P}(f_{c_0})$ when $r \nearrow \infty$. However, for intermediate values $0<r<R<\infty$, the dependence is quite complex. Below, we point out some tractable features in the evolution of ${\cal P}(f)$ for the case of when the transition radius $R$ is small and $r \nearrow R$, then for the case of a point-wise mutation at the origin ($r=0$) and $R \nearrow \infty$.

\subsection{Small transition radius}
\label{small}

\noindent Figure~\ref{parameter_cases} shows a few examples of how pointwise mutations with different transition radii $R$ affect the asymptotic dynamics of an intact map (with prisoner set the Douady rabbit). The third row (for smaller $R$) suggests that if the mutation, including its transient region, is not too extensive (i.e., $R$ is small), then the prisoner set for the mutated iteration remains a subset of the prisoner set obtained for the intact map. A sufficient condition expressing how small $R$ needs to be for this to occur depends on the geometry of the intact prisoner set (i.e., on the system's behavior in the absence of the mutation). More precisely, if $\text{D}(R)$ is the disc centered at the origin with radius $R$, we have the following:

\begin{lemma}
If the radius $R$ is such that $\text{D}(R) \subseteq {\cal P}(f_{c_1})$, then ${\cal P}(f) \subseteq {\cal P}(f_{c_1})$.
\label{inclusion_lemma}
\end{lemma}

\proof{Suppose $z_0 \in {\cal P}(f)$. Call $z_n = f_{c_1}^{\circ n}(z_0)$, for all $n \geq 1$. If there exists an $N \geq 0$ such that $| z_N | < R$, then $z_N \in {\cal P}(f_{c_1})$, and the forward orbit of $z_N$ under $f_{c_1}$ is bounded, hence the sequence $(z_n)_{n \geq 1}$ is bounded. If $| z_n | > R$ for all $n \geq 1$, then $f^{\circ n}_{c_1}(z_0) = f^{\circ n}(z_0)$ for all $n \geq 1$. Since the latter is bounded, it follows that $(z_n)_{n \geq 1}$ is also bounded. Either way, $z_0 \in {\cal P}(f_{c_1})$.

\qed }\\

\begin{figure}[h!]
\begin{center}
\includegraphics[width=0.7\textwidth]{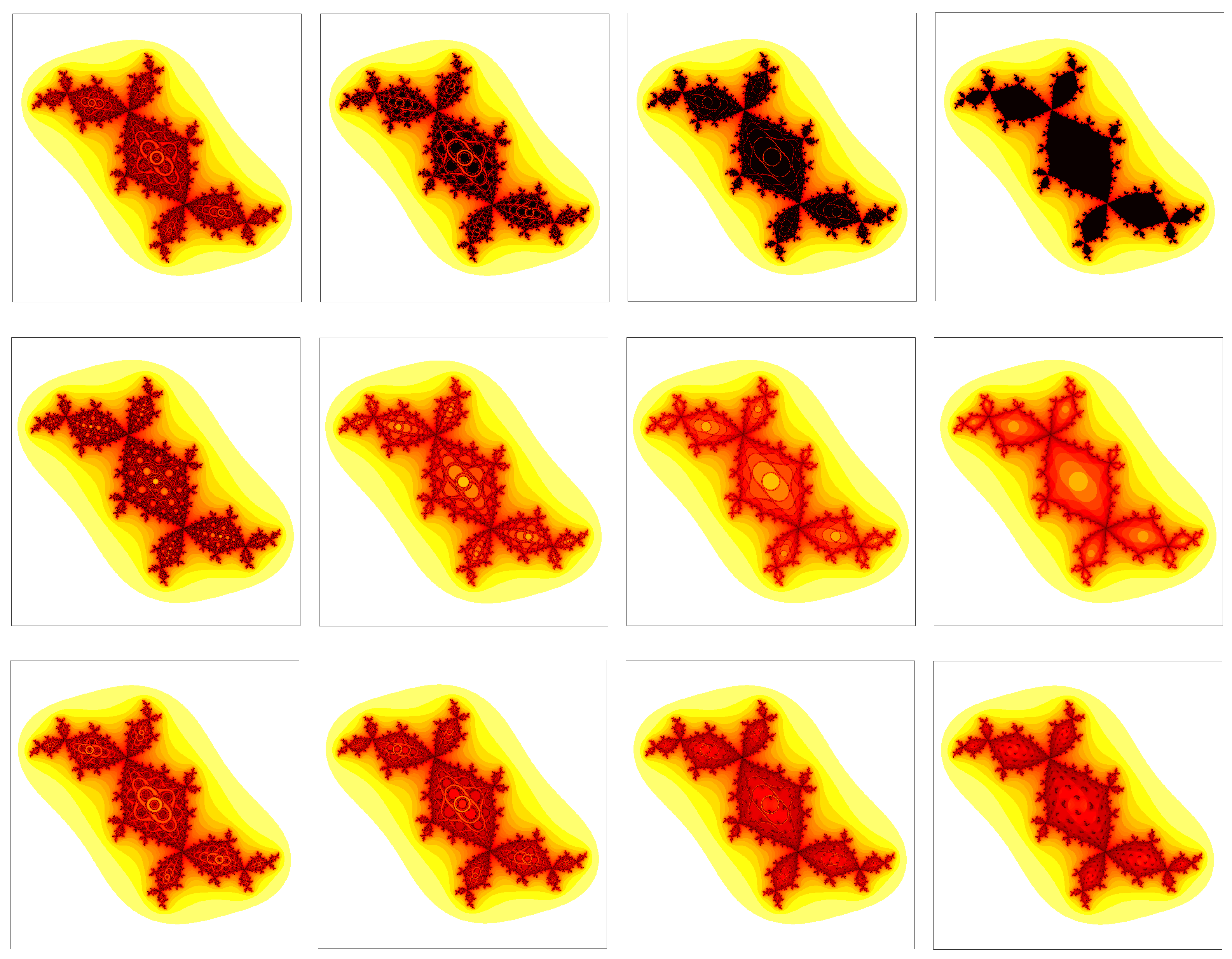}
\end{center}
\vspace{-4mm}
\caption{\small \emph{{\bf Dependence of prisoner set on the mutation radius $r$ when $D(R) \subseteq {\cal P}(f_{c_1})$.} All panels represent mutations with fixed transition radius $R=0.1$, acting at the origin on $c_1= -0.13+0.77i$. From top to bottom, each row corresponds to a different mutant $c_0$, as follows: {\bf Top:} $c_0=0$. {\bf Middle:} $c_0=-0.13-0.77i$. {\bf Botom:} $c_0=0.33$. The mutation radius is increased from left to right, so that each column corresponds respectively to: $r=0$; $r=0.04$; $r=0.08$; $r=0.1=R$.}}
\label{increasing_r}
\end{figure}

\noindent Notice that with the smallness restriction for $R$ which ensures that ${\cal P}(f) \subseteq {\cal P}(f_{c_1})$, the prisoner set of the mutated system ${\cal P}(f)$ may still be connected or break into multiple connected components, the number, size and geometry of which depend on both parameters $(c_0,c_1)$ and radii $(r,R)$. What happens more specifically along the path of increasing $r$ from zero to $R$ depends on the context. 

For example, the top row in Figure~\ref{increasing_r} illustrates the situation corresponding to the mutation $c_0=0$ (the unit disc) applied at the origin to the Douady rabbit $c_1= -0.13+0.77i$, as $r$ increases from zero to the fixed transition radius $R=0.1$. In this case, we have that ${\cal P}(f)$ approaches ${\cal P}(f_{c_1})$ as $r \nearrow R$, in the following sense: for all $z \in {\cal P}(f_{c_1})$, there exists $\delta > 0$ such that $z \in {\cal P}(f)$ for all $r>R-\delta$, where $f$ is the mutated map corresponding to the mutation radius $r$. This is in fact more generally the case for any $c_1$ that represents the center of a hyperbolic component of the Mandelbrot set, and transition radius $R$ such that $D(R) \subset {\cal P}(f_{c_1})$. Indeed, notice first that, since the Julia set of $f_{c_1}$ is invariant under iterations of $f_{c_1}$, any $z \in {\cal J}(f_{c_1}) = \partial {\cal P}(f_{c_1})$ iterates back to ${\cal J}(f_{c_1}) \subset \mathbb{C} \setminus D(R)$, hence will have the same bounded trajectory under both $f_{c_1}$ and $f$. Consider now $z=z_0$ in the interior $\overset{\circ}{\cal P}(f_{c_1})$. Since all interior prisoners are attracted to the superstable critical orbit, there is a minimal $N \geq 0$ such that $f^{\circ N}(z_0) \in D(R)$. The minimality condition implies that $\lvert f^{\circ k}(z_0) \rvert \geq R$, hence $f^{\circ k}(z_0) = f_{c_1}^{\circ k}(z_0)$, for all $k<N$ (in case $N>0$). Now, take $\delta < R-\lvert f^{\circ N}(z_0) \rvert$. Then, for $r > R-\delta$, we have that $\lvert f^{\circ N}(z_0) \rvert < r$, hence $\lvert f^{\circ n}(z_0) \rvert < r$ for all $n \geq N$ (since the map $f$ acts as $z \to z^2$ inside $D(r)$, as illustrated in Figure~\ref{proof}a). This concludes that $z=z_0 \in {\cal P}(f)$, for $f$ constructed with the given mutation radius $r$. Notice that the condition for $r$ depends on the choice of $\delta$, hence on the point $z$.

\begin{figure}[h!]
\begin{center}
\includegraphics[width=0.75\textwidth]{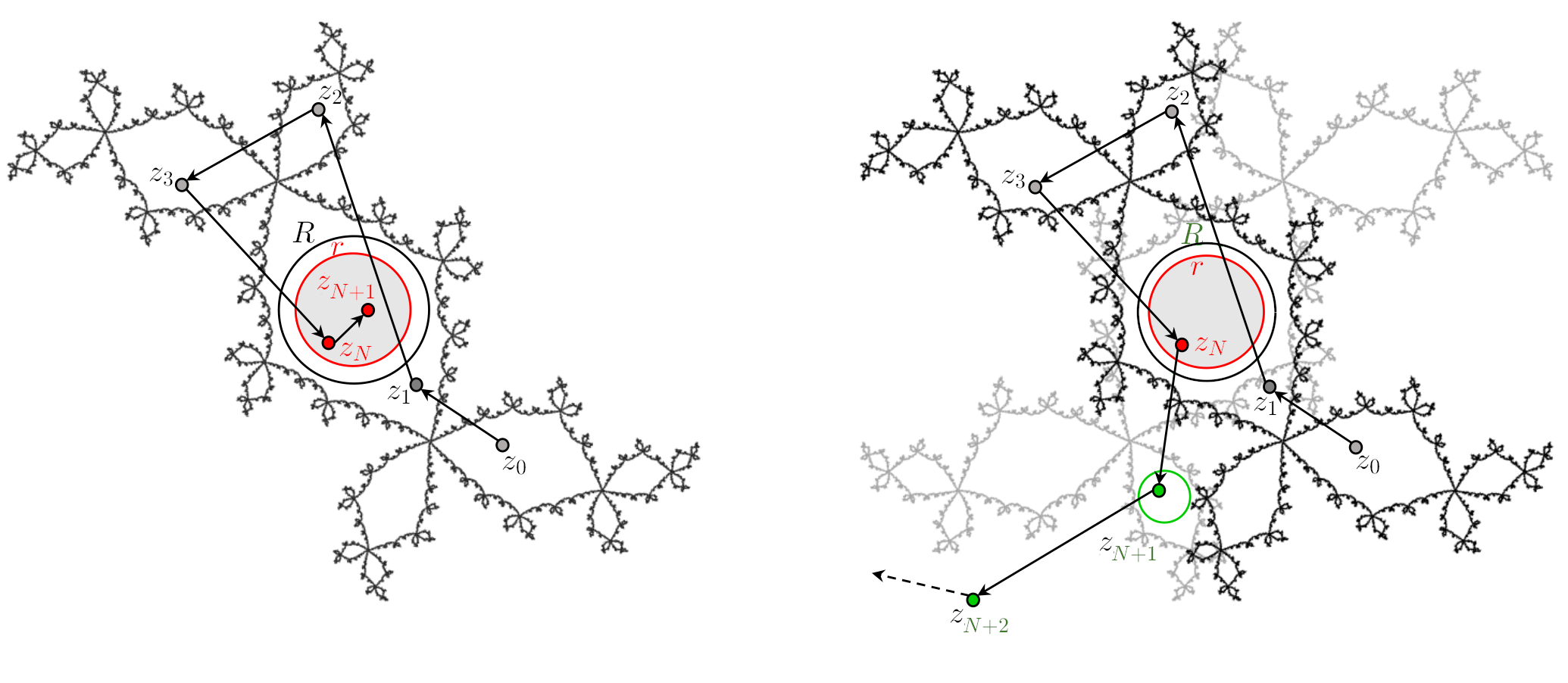}
\end{center}
\vspace{-4mm}
\caption{\small \emph{{\bf Examples of limit behavior of ${\cal P}(f)$ as $r \nearrow R$.} {\bf Left.} $c_1= -0.13+0.77i$, $c_0=0$. The transition radius $R$ is represented by the black circle, and is taken such that the disk $D(R)$ is a subset of the Douady rabbit ${\cal P}(f_{c_1})$ (shown as a contour). The orbit of an arbitrary point $z_0$ in $\overset{\circ}{\cal P}(f_{c_1})$, is sketched symbolically up to $z_N = f^{\circ N}_{c_1}(z_0)$, its first iterate that falls in the interior of $D(R)$. Then the mutation radius (represented by the red circle) can be chosen so that $\lvert z_N \rvert < r <R$, and $z_N \in D(r)$. The rest of the trajectory $z_{n}=f^{\circ n}(z_0)$, $n \geq N+1$ is then iterated under $f_{c_0}(z) = z^2$, and is therefore trapped in $D(r)$. {\bf Right.}  $c_1= -0.13+0.77i$, $c_0=-0.13-0.77i$. The contour of the Douady rabbit $c_1$ is shown in black, and that of the inverted rabbit $c_0$ is shown in gray. The transition radius $R$ is represented by the black circle, and is small enough so that $D(R) \subset {\cal P}(f_{c_1})$, and so that the first iterate of $D(R)$ (contained inside the green circle) is outside of  ${\cal P}(f_{c_1})$, such that $f_{c_1}(D(R)) \cap {\cal P}(f_{c_0}) = \phi$. The orbit of an arbitrary $z_0 \in \overset{\circ}{\cal P}(f_{c_1})$ is shown symbolically up to its first iterate $z_N \in D(R)$. Then the mutation radius $r$ (red circle) can be chosen such that $\lvert z_N \rvert < r <R$, and $z_{N+1} = f(z_N)$ is outside of both $D(R)$ and ${\cal P}(f_{c_1})$, hence it escapes under $f$.}}
\label{proof}
\end{figure}

On the other hand, the middle row in Figure~\ref{increasing_r} shows how the mutation $c_0=\overline{c_1}$ leads to a different outcome when applied to the same intact function $c_1= -0.13+0.77i$, as $r$ increases towards the same transition radius $R=0.1$. In this case, ${\cal P}(f)$ approaches the Julia set ${\cal J}(f_{c_1}) = \partial {\cal P}(f_{c_1})$ as $r \nearrow R$. This is true because $R$ is small enough so that its first image under the mutated map $f_{c_0}$ does not intersect the prisoner set of the intact map $f_{c_1}$, that is: $f_{c_1}(D(R)) \subset {\cal P}(f_{c_1}) \setminus {\cal P}(f_{c_0})$ (see Figure~\ref{proof}b). Since the Julia set ${\cal J}(f_{c_1}) \subset \mathbb{C} \setminus D(R)$ is invariant under $f_{c_1}$, it follows that it is also invariant under $f$, hence ${\cal J}(f_{c_1}) \subset {\cal P}(f)$. As before, for any interior point $z = z_0 \in \overset{\circ}{\cal P}(f_{c_1})$, there is a first iterate $z_N = f_{c_1}^{\circ N}(z_0) \in D(R)$. Taking $\delta < R- \lvert z_N \rvert$, and $r > R-\delta$, it follows that $z_N \in D(r)$, hence $f(z_N) = f_{c_0}(z_N) \subset f_{c_0}(D(R)) \subset \mathbb{C} \setminus {\cal P}(f_{c_1})$, and $z_N \notin D(R)$, which means that $z=z_0$ will escape under both $f_{c_1}$ and $f$. Hence every interior prisoner of ${\cal P}(f_{c_1})$ will eventually be excluded from ${\cal P}(f)$, as $r \nearrow R$. 

The bottom row of Figure~\ref{increasing_r} illustrates the case where the mutation $c_0=0.33$ is introduced at the origin into the Douady rabbit $c_1= -0.13+0.77i$, for the same fixed transition radius $R=0.1$. The panels show that ${\cal P}(f)$ approaches a proper subset of ${\cal P}(f_{c_1})$ which contains interior points of ${\cal P}(f_{c_1})$, as $r \nearrow R$.

\noindent The situation changes qualitatively when the radius of the transient annulus is larger than prescribed in the lemma, as shown in the bottom row of Figure~\ref{parameter_cases}, and additionally illustrated in Figure~\ref{increasing_r(b)}.  If the smallness condition for $R$ in Lemma~\ref{inclusion_lemma} no longer applies, then ${\cal P}(f)$ can no longer be expected to remain a subset of ${\cal P}(f_{c_1})$. In fact, we conjecture that the opposite is true, under general assumptions for the maps and for the two radii (see also Figure~\ref{increasing_r(b)_sketch}):

\begin{figure}[h!]
\begin{center}
\includegraphics[width=0.75\textwidth]{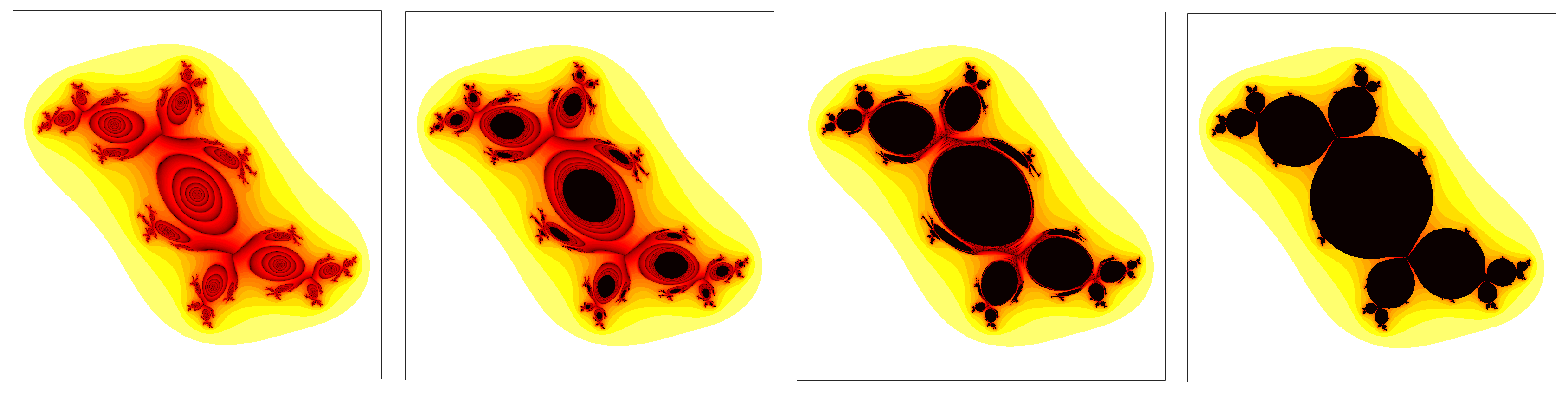}
\end{center}
\vspace{-4mm}
\caption{\small \emph{{\bf Dependence of prisoner set on the mutation radius $r$ when $D(R) \setminus {\cal P}(f_{c_1}) \neq \phi$.} All panels represent the mutation $c_0=0$ acting at the origin on $c_1= -0.13+0.77i$, with fixed transition radius $R=0.5$. The mutation radius is increased in panels (I) to (IV) as follows: $r=0$; $r=0.1$; $r=0.3$; $r=0.5$.}}
\label{increasing_r(b)}
\end{figure}

\begin{conj}
If the radius $R$ is such that ${\cal P}(f) \subseteq {\cal P}(f_{c_1})$, then $\text{D}(R) \subseteq {\cal P}(f_{c_1})$.
\end{conj}

\begin{figure}[h!]
\begin{center}
\includegraphics[width=0.9\textwidth]{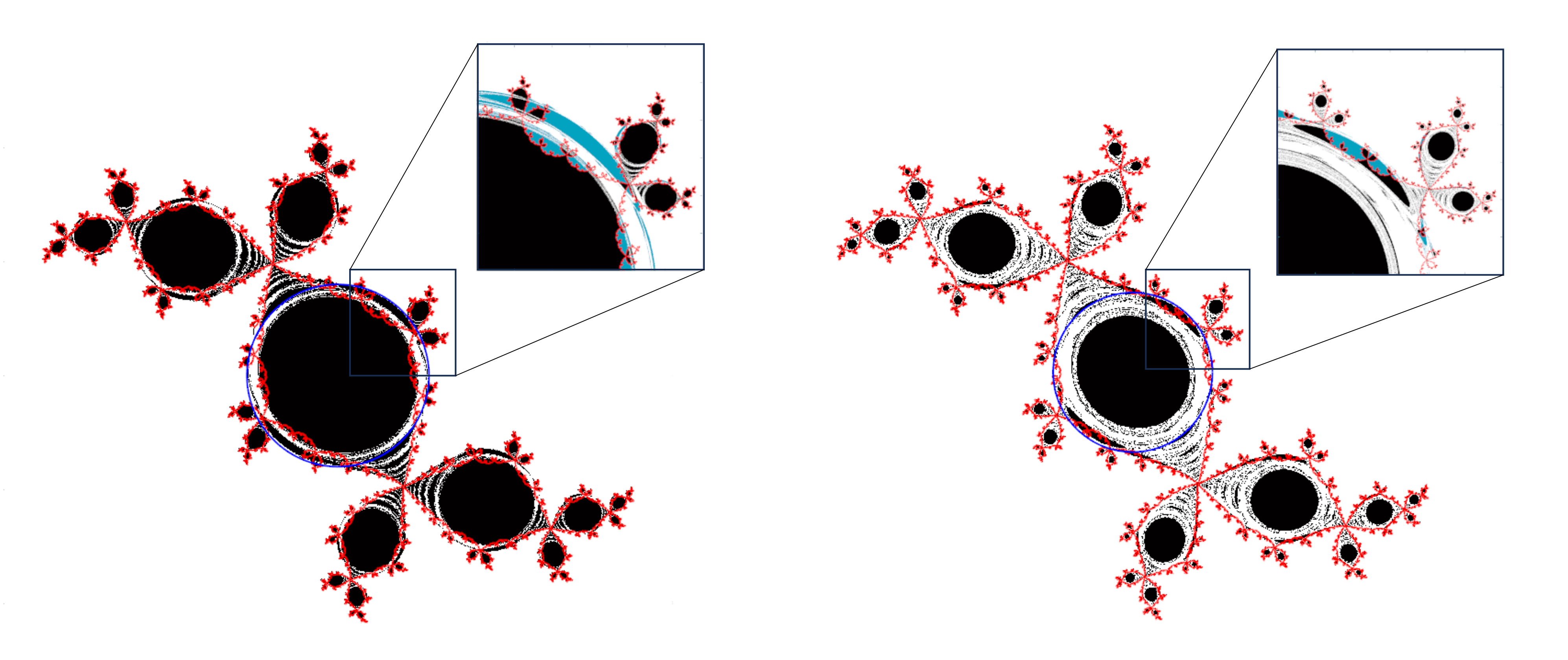}
\end{center}
\vspace{-4mm}
\caption{\small \emph{{\bf Illustration of prisoner sets ${\cal P}(f)$} that are not subsets of ${\cal P}(f_{c_1})$, when $D(R)$ is not a subset of ${\cal P}(f_{c_1})$. Both panels show the prisoner set for a mutation $c_0=0$ at the origin, inserted into the intact map $c_1= -0.13+0.77i$ (Douady rabbit), with mutation radius $r=0.3$ and transition radius $R=0.4$ (left) and with mutation radius $r=0.2$ and transition radius $R=0.35$ (right). The Julia set for the intact map $c_1$ is also shown in red, and the circle bounding $D(R)$ is shown in dark blue in each case. In each panel, a small square region is magnified in the respective insert, to better illustrate a subset pf points in ${\cal P}(f) \setminus {\cal P}(f_{c_1})$ (showed as a shaded blue region).}}
\label{increasing_r(b)_sketch}
\end{figure}

A mutation with a small locus $R$ can be interpreted as corresponding to an error in replication which only affects a small number of genes. More precisely, for a permanent mutation as the ones we are considering in our model, some genes in the small locus $D(R)$ which should be expressed in the intact replication are dropped by the perturbed replicator, and/or conversely, some wrong genes may be retained in the expression at certain time steps. Interestingly, what our results suggest is that, as long as the mutation affects a small enough locus, the end result will be that of missing some of the intact gene expression in the terminal differentiation. This may explain, for example, certain loss of function in the resulting cells, such as the reduced neurogenesis documented in~\cite{puigdevall2023somatic}. On the other hand, a mutation with larger radius $R$ will manifest itself with both omission and addition of inaccurate gene expression to the differentiation locus. This may be the case of oncogenic mutations (associated in some empirical studies with the presence of ``de-differentiation'' stages ~\cite{carvalho2020cell}), which may manifest as exacerbation of certain functions, such as abnormal cell proliferation.


\subsection{Large transition radius} 
\label{large}

A mutation with very large transition radius $R$ corresponds to an iterated map where the mutation $f_{c_0}$ and its ripple effects $f_{01}$ act on most of the DNA expression. It would be logical for the gene expression in the terminally differentiated cell under the mutation to be more similar to that of a cell obtained by iterations of the erroneous replicator $f_{c_0}$ than of the intact replicator $f_{c_1}$ (the action domain of which is pushed towards infinity, thus becoming virtually irrelevant). This may be, for example, what occurs in the differentiation process of epithelial cells in acute burn injury. Empirical studies revealed increased terminal differentiation of epithelium cells called keratinocytes in non-burned skin allografts, while cells from burn injury subjects showed significantly higher proliferative profile and genetic expression enhancing the physiologic wound healing process~\cite{gauglitz2012functional}. This suggests that, for the duration of the healing, a different genetic differentiation algorithm may take over a relatively large RNA expression locus, virtually leading to a different type of terminally differentiated cells under the changed environment~\cite{evans2013epithelial}.

It is then interesting to formalize mathematically in which way the prisoner set of the mutated system ${\cal P}(f)$ approaches ${\cal P}(f_{c_0})$, as the mutation radius increases. In Figure~\ref{increasing_R(a)}, we illustrate the evolution of the prisoner set of the iteration, as the transition radius $R$ increases from $R=0$ to $R \to \infty$, for the point-wise mutation: $c_0=0$. In Figure~\ref{increasing_R(b)}, we show this evolution for one additional point-wise mutation: $c_0=-0.117-0.856i$. There are similarities and differences between the two pathways. For any combination of maps, the prisoner set is originally (i.e., for $r=R=0$) identical almost everywhere with the prisoner set of the intact map $f_{c_1}$ (the combinatorics of the iteration are identical with those under $f_{c_1}$, except on the critical orbit of $f_{c_0}$, which is at most countable). Increasing $R$ from zero puts ${\cal P}(f)$ through a sequence of topological transformations, eventually leading towards the prisoner set of the mutation map $f_{c_0}$. Below, we prove a convergence result that formalizes in which sense ${\cal P}(f)$ approaches ${\cal P}(f_{c_0})$ as $R$ gets large. 

\begin{thm} Fix an arbitrary pair $(c_0,c_1) \in \mathbb{C}^2$ and an $N \in \mathbb{N}$. For any $z_0 \in \mathbb{C}$ and for any $\varepsilon>0$, there exists $K=K(z_0,\varepsilon)>0$ such that, if $R>K$, then $\lvert f^{\circ n}(z_0) - f^{\circ n}_{c_0}(z_0) \rvert < \varepsilon$ for all $n \leq N$, where $f$ is the mutated map corresponding to a point-wise mutation $f_{c_0}$ applied at the origin and with transition radius $R$.
\label{convergence_theorem}
\end{thm}

\noindent {\bf Remark.} The theorem shows that, for any pair of intact map and mutation, one can control the distance between the trajectory of any point under the mutated map $f$ and the mutation itself $f_{c_0}$ over a given (finite, but possibly large) number $N$ of iterates, by allowing $R$ to be sufficiently large. 

\begin{figure}[h!]
\begin{center}
\includegraphics[width=0.7\textwidth]{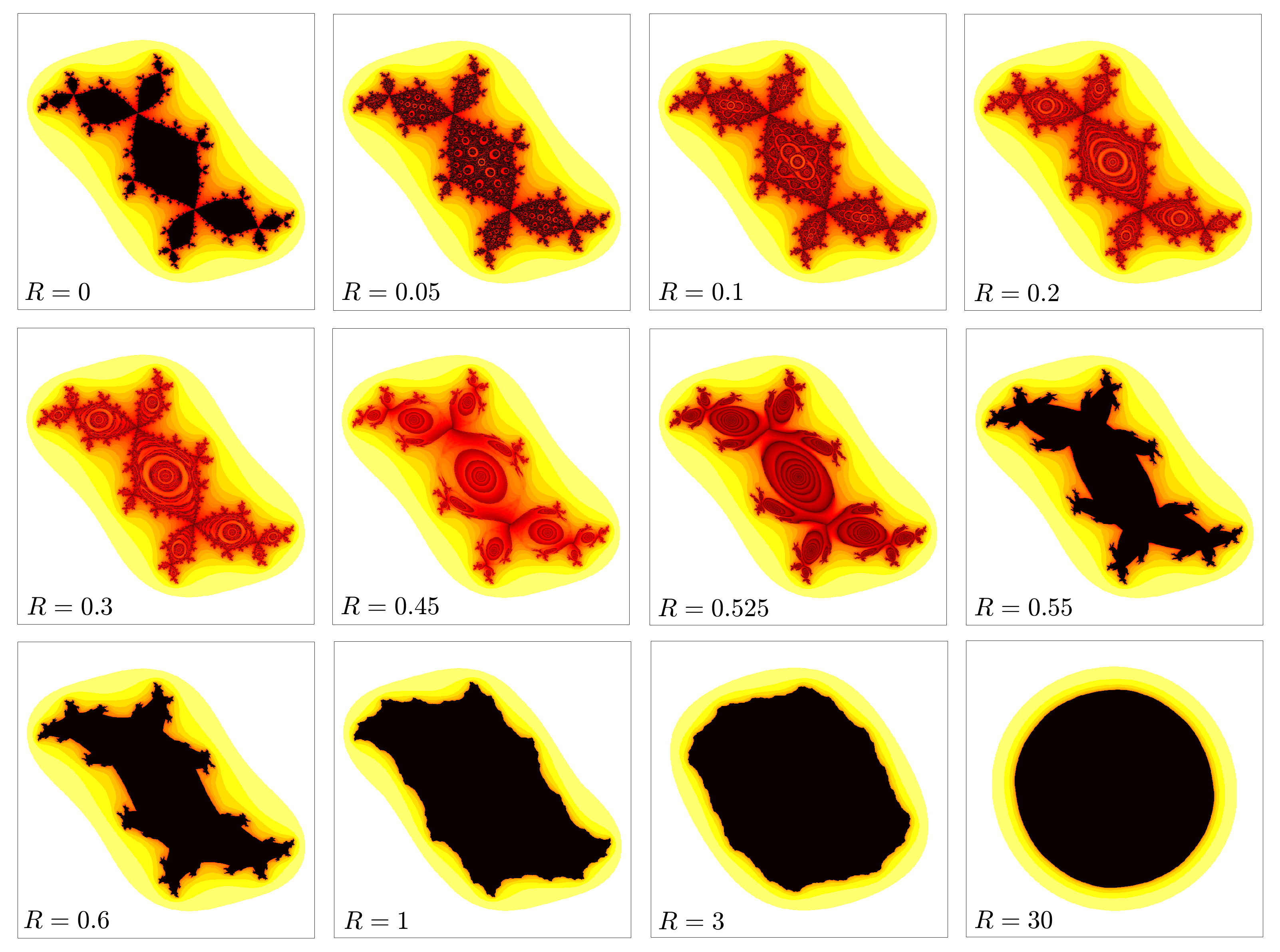}
\end{center}
\vspace{-3mm}
\caption{\small \emph{{\bf Dependence of the prisoner set on the transition radius $R$.} All panels represent the point-wise mutation $c_0=0$ acting at the origin on $c_1= -0.13+0.77i$. The transition radius is increased in panels (I) to (XII) (left to right then top down) as follows: $R=0$; $R=0.05$; $R=0.1$; $R=0.2$; $R=0.3$; $R=0.45$; $R=0.525$; $R=0.55$; $R=0.6$; $R=1$; $R=3$; $R=30$. }}
\label{increasing_R(a)}
\end{figure}

\begin{figure}[h!]
\begin{center}
\includegraphics[width=0.7\textwidth]{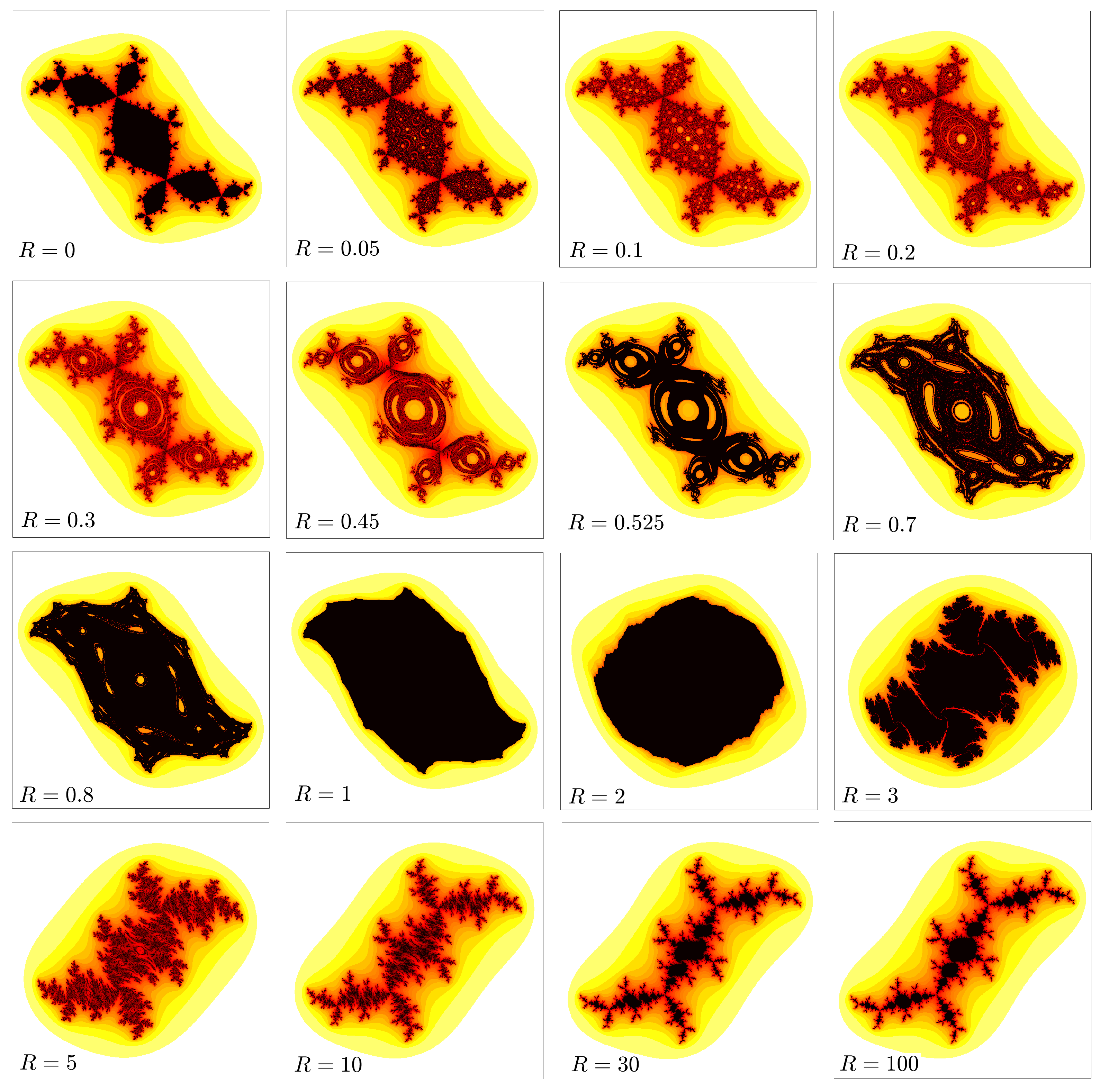}
\end{center}
\vspace{-3mm}
\caption{\small \emph{{\bf Dependence of prisoner set on the transition radius $R$.} All panels represent the point-wise mutation $c_0=-0.117-0.856i$ acting at the origin on $c_1= -0.13+0.77i$. The transition radius is increased in panels (I) to (XVI) (left to right, then top down) as follows: $R=0$; $R=0.05$; $R=0.1$; $R=0.2$; $R=0.3$; $R=0.45$; $R=0.525$; $R=0.7$; $R=0.8$; $R=1$; $R=2$; $R=3$; $R=5$; $R=10$; $R=30$; $R=100$. }}
\label{increasing_R(b)}
\end{figure}

\proof{Fix the map pair $(c_0,c_1) \in \mathbb{C}^2$, and a number of iterates $N \in \mathbb{N}$, and let $\varepsilon >0$. For any $z_0 \in \mathbb{C}$, we consider the finite set of iterates $\{ z_n = f^{\circ n}_{c_0}(z_0), n \leq N \}$, and $\{ \xi_n = f^{\circ n}(z_0), n \leq N \}$. Let $M$ be an upper bound $M = M(z_0) > \max\{ \lvert z_n \rvert, \lvert \xi_n \rvert, n \leq N \}$. First, notice that we have:
\begin{eqnarray*}
\lvert \xi_1 - z_1 \rvert = \lvert f(z_0) - f_{c_0}(z_0) \rvert &=& \left \lvert  z_0^2 + c_0 \frac{R-\lvert z_0 \rvert}{R} + c_1 \frac{\lvert z_0 \rvert}{R} -z_0^2-c_0 \right \rvert \\
&=& \lvert c_1-c_0 \rvert \frac{\lvert z_0 \rvert}{R} < \frac{\lvert c_1 - c_0 \rvert M}{R}
\end{eqnarray*} 

\noindent 
We prove inductively that, for all $n \geq 1$, we have that $\ds \lvert \xi_n - z_n \rvert < \frac{\lvert c_1 - c_0 \rvert Q_n(M)}{R}$, where $Q_n$ are polynomials with positive coefficients which do not depend on $R$. Clearly, the statement is true for $n=1$. Suppose it is true for $n$. We then have:
\begin{eqnarray*}
\lvert \xi_{n+1} - z_{n+1} \rvert &=& \lvert f(\xi_n) - f_{c_0}(z_n) \rvert = \left \lvert  \xi_n^2 + c_0 \frac{R-\lvert \xi_n \rvert}{R} + c_1 \frac{\lvert \xi_n \rvert}{R} -z_n^2-c_0 \right \rvert \\\\
&=& \lvert \xi_n^2 - z_n^2 \rvert + \lvert c_1-c_0 \rvert \frac{\lvert \xi_n \rvert}{R} < 2M \lvert \xi_n-z_n \rvert+ \lvert c_1-c_0 \rvert \frac{M}{R}\\
&<& \frac{\lvert c_1-c_0 \rvert}{R} [2MQ_n(M)+M]
\end{eqnarray*} 

\noindent Hence the statement holds for $n+1$, with $Q_{n+1}(M) = 2MQ_n(M)+M$. To conclude the proof of the theorem, one can take $K(z_0,\varepsilon) = \frac{\lvert c_1-c_0 \rvert}{\varepsilon}\max \{ Q_n(M), 0 \leq n \leq N \}$. Clearly, for $R>K$, we have that $\lvert \xi_n - z_n \rvert < \varepsilon$, for all $0 \leq n \leq N$.
}

\begin{corol} For any $\epsilon>0$, the value of $K$ in Theorem~\ref{convergence_theorem} can be chosen so that it works uniformly for all initial conditions $z_0 \in {\cal P}(f_{c_0})$.
\end{corol}

\begin{proof} The orbits $z_n = f^{\circ n}_{c_0}(z_0)$ of all prisoners $z_0 \in {\cal P}(f_{c_0})$ are uniformly bounded by the escape radius $R_e = \max \{2,\lvert c_0 \rvert \}$. As before, we have that $\lvert \xi_1 - z_1 \rvert < \frac{\lvert c_1 - c_0 \rvert R_e}{R}$, and we will show inductively that
\begin{eqnarray*}
\lvert \xi_n - z_n \rvert < \frac{\lvert c_1 - c_0 \rvert}{R} T_n \left(R_e,1/R\right)
\end{eqnarray*} 

\noindent  for all $n \geq 1$, where $T_n$ are polynomials with positive coefficients which do not depend on $z_0$. Clearly the statement is true for $n=1$. Assume it is true for $n$; then we subsequently have that 
$$\lvert \xi_n \rvert <  \frac{\lvert c_1 - c_0 \rvert}{R} T_n(R_e,1/R) + \lvert z_n \rvert <  \frac{\lvert c_1 - c_0 \rvert}{R} T_n(R_e,1/R) + R_e$$

\noindent It follows that\\

$\ds \lvert \xi_{n+1} - z_{n+1} \rvert \leq (\lvert z_n \rvert + \lvert \xi_n \rvert) \lvert \xi_n-z_n \rvert+ \lvert c_1-c_0 \rvert \frac{\lvert \xi_n \rvert}{R}$\\\\

$\ds \quad \quad \quad < \left[ 2R_e +  \frac{\lvert c_1 - c_0 \rvert}{R} T_n \left(R_e, \frac{1}{R} \right) \right]  \frac{\lvert c_1 - c_0 \rvert}{R} T_n \left( R_e, \frac{1}{R} \right) + \frac{\lvert c_1-c_0 \rvert }{R} \left[  \frac{\lvert c_1 - c_0 \rvert}{R} T_n \left( R_e,\frac{1}{R} \right)+ R_e \right]$\\\\
 
$\ds \quad \quad \quad =  \frac{\lvert c_1 - c_0 \rvert}{R} T_{n+1} \left(R_e,\frac{1}{R} \right)$\\

\noindent where
$$T_{n+1} = \left[ 2R_e + \frac{\lvert c_1 - c_0 \rvert}{R} T_n \left(R_e, \frac{1}{R} \right) \right] T_n \left( R_e, \frac{1}{R} \right) +  \frac{\lvert c_1 - c_0 \rvert}{R} T_n \left( R_e,\frac{1}{R} \right)+ R_e $$

\noindent This concludes the induction, and subsequently shows that $R$ can be made large enough (in a manner that depends on $\varepsilon$ and $R_e$, but not on the choice of $z_0 \in {\cal P}(f_{c_0})$) so that $\lvert \xi_n -z_n \rvert < \varepsilon$, for all $0 \leq n \leq N$.

\end{proof}

\begin{corol} If $z_n = f_{c_0}^{\circ n}(z_0)$ escapes, then $\xi_n = f^{\circ n}(z_0)$ also escapes, for large enough $R$.
\end{corol}

\begin{proof} The result follows directly from Theorem~\ref{convergence_theorem}, and the existence of an escape radius. 
\end{proof}

\noindent In the path from ``almost'' ${\cal P}(f_{c_1})$ (for $R \searrow 0$) to ``almost'' ${\cal P}(f_{c_0})$ (for $R \nearrow \infty$), the prisoner set undergoes a complex cascade of bifurcations, which depends on the structure of the original and of the target sets. Our simulations suggest that, as $R$ increases from zero, ${\cal P}(f)$ breaks into a number of connected components; the original bifurcation (from a single to multiple connected components) may occur close to $R=0$ or later in the process (as illustrated in Figure~\ref{increasing_R(a)} for $c_1= -0.13+0.77i$ and $c_0=0$, and in Figure~\ref{increasing_R(b)}, as well as in Figures~\ref{center_1} and~\ref{center_diag} in the Appendix). These components may later re-merge into one, either via another kneading cascade, or via one single sharp phase transition. However, depending on the case, this connected set may undergo a further connectivity altering bifurcations, in order to finally converge to the target set ${\cal P}(f_{c_0})$, in the sense described in Theorem~\ref{convergence_theorem}.

%
%
%
%
%
%
%


\section{Discussion}
\label{discussion}

\subsection{Specific comments on the model}

This paper lays down the bases of a symbolic context for analyzing the effect of local mutations in a genetic replication process. In this setup, the complex plane $\mathbb{C}$ represents the expression of the features to be replicated (genes in a cell's RNA). A correct replication would consist of repeatedly applying the intact complex map $f_{c_1} = z^2 + c_1$ to the complex plane, resulting each time into a new copy of $\mathbb{C}$, used as template for subsequent replications. A local mutation is inserted into the iteration, acting as an erroneous function $f_{c_0}$ within a small local disk around a focus point in $\mathbb{C}$, and continuously interpolating to $f_{c_1}$ within a ``transition'' annulus around the mutation disk. We analyzed potential behaviors of the asymptotic trajectories under iterations of the resulting mutated system versus intact replication. Our study focused on observing and analyzing the effects of the size of the mutation and of the transition radii on the topology of the prisoner set. Some of these effects were intuitively straightforward; others were more complex and counter-intuitive. Here, we discuss the results and further interpret them within the phenomenology of genetic differentiation and cell functional specialization in an organism.

Our analysis showed that a point-wise mutation with very small transition annulus $r=0$ and $R \searrow 0$) will not make any noticeable changes to the asymptotic dynamics (the prisoner set is the same as the intact prisoner set almost everywhere), while a mutation with a large enough transition radius will eventually represent asymptotically the mutated replicator, rather than the intact one (the prisoner set converges to that of the mutation when $R$ approaches infinity). In reality, this corresponds to a model prediction that mutation of only a few genes, even if it occurs repeatedly, will not significantly affect gene expression in the differentiated cell, while mutation to a wide genetic locus will cause the terminal differentiation of the cell to primarily represent the mutated clones.

More interesting aspects appear when tracking the transition from mutations that affect a very small to a very large genetic locus (i.e., from $R \searrow 0$ to $R \nearrow \infty$). While the particular effects depend widely on the type of replicator and mutation (that is, on the complex parameters $c_1$ and $c_0$), we found that all paths along increasing $R$ exhibit phase transitions. In other words, the evolution of the prisoner set, rather than moving smoothly with increasing $R$, undergoes bifurcation points characterized by sudden changes in topology, in particular in the number of connected components. There are a few interesting implications. An immediate note would be related to the very idea of bifurcation: while some increments in $R$ may leave the asymptotic outcome relatively unperturbed, it also appears that there are critical values of $R$ where even a very small additional increment has the potential to dramatically change the prisoner set.

Moreover, these bifurcation cascades do not typically start at $R=0$, but rather at the critical $R=R_0 >0$ where $D(R)$ breaks out of ${\cal P}(f_{c_1})$ (e.g., for $c_1=-1.13+0.77i$, the first connectivity braking point occurs at $R_0 \sim 0.3$). This implies the existence in these cases of a ``safe'' interval $[0,R_0]$ in which the prisoner set ${\cal P}(f) \subseteq {\cal P}(f_{c_1})$ changes smoothly with $R$ and is still generally reminiscent of the intact ${\cal P}(f_{c_1})$. In other words, mutations which affect a small enough genetic locus will only manifest as omissions in the terminal differentiation locus of the cell. Interestingly, the subset of omitted genes does not necessarily increase with increasing $R$: within this range, a larger mutation locus does not necessarily imply a harder impact on the eventual cell outcome. In fact, this dependence lacks ``monotonicity'' along $R \nearrow \infty$, in the sense that, for some values of $R$, the prisoner set is more similar to the intact ${\cal P}(f_{c_1})$ than it is for both smaller and larger values around (as for example panel (VIII) in Figure~\ref{increasing_R(a)}). Varying the size of the mutation locus itself (that is, the size of $r$) further modulates these effects.


\subsection{Future work}

Our first results suggest a strong dependence of the prisoner set on the mutation error (the local parameter $c_0$), in combination with the intact parameter $c_1$ that should ideally govern the correct replication process along the cell's specialization path. The traditional theory of single iterated quadratic maps implies that different replicators lead to a large variety of prisoner sets. What we would like to quantify and investigate further is the way in which the geometric perturbation of the prisoner set becomes more significant with a larger mutation (i.e., a larger $|c_1 - c_0|$). We can then analyze how this effect further depends on the replicator $c_1$ itself.

Another short term goal in our future work is to extend this model to the situation where the mutation locus is centered at a focus different than the origin. This is important, because it would allow investigating the effects of mutations that affect expression of different RNA segments. Mathematically, this extension perturbs the symmetry of the simpler model (with the mutation focus at the origin) and the effects to the prisoner set will reflect this symmetry braking, as illustrated in Figures~\ref{center_1} and~\ref{center_diag} in the Appendix.

While the current model, describing a persistent (or permanent) local mutation, is both simpler mathematically and consistent with the dynamics of a system generated by a single iterated map, in reality mutations do not affect expression of the same RNA targets at each and every step of the replication. Moreover, in an actual dividing cell, errors in RNA replication are sometimes immediately addressed by molecular repair processes, which in our model may be considered equivalent to reversal to the correct iteration when a mutation is detected. Future extensions of our model will consider the mutation being introduced at selected (specific or random) times in the iteration process, allowing us to investigate the effects of timing of mutation onset and repair to the accuracy of replication and terminal differentiation of the cell. This would place the model within the mathematical framework of non-autonomous iterations, which has been previously studied by the authors in the context of template iterations, in absence of mutations~\cite{radulescu2016symbolic,radulescu2016template}.

Another aspect we plan to further study in future work relates to potential applications to modeling neoplastic processes and tumor growth. It has been known for some time that combinations of multiple mutations, rather than a single mutation, might play more important roles in the genesis of tumors~\cite{park2015cancer,ashworth2011genetic}. However, the cooperative mechanisms by which accumulation of multiple somatic mutations lead to cancer and tumor growth are poorly understood. Newer work emphasizes the importance of investigating the formation and evolution of dynamic clusters of mutation-propagating modules. For example, a study of colorectal tumorigenesis found that the mutation network undergoes a phase transition from scattered small modules to a large connected cluster, accompanied by phenotypic changes corresponding to cancer hallmarks~\cite{shin2017percolation}. Our mathematical approach presents great potential for further investigating new hypotheses of cluster dynamics in a theoretical modeling framework. To do so, we would have to allow the radius $R$ (which can be viewed to define the boundaries of the cluster) to evolve in time, and interact with the rest of the network, as illustrated in Figure~\ref{increasing_R}.

\begin{figure}[h!]
\begin{center}
\includegraphics[width=0.75\textwidth]{Figure9.png}
\end{center}
\vspace{-4mm}
\caption{\small \emph{{\bf Prisoner sets for transition radius $R$ increasing in time} as $R=0.05n$, where $n$ is the iteration step. Each panel represents a different point-wise mutation at the origin, as follows: {\bf A.} $c_0=0$, $c_1= -0.13+0.77i$; {\bf B.} $c_0=-0.117-0.856i$, $c_1= -0.13+0.77i$; {\bf C.} $c_0=0.33$, $c_1= -0.13+0.77i$; {\bf D.} $c_0= -0.13+0.77i$, $c_1=0.33$.}}
\label{increasing_R}
\end{figure}

In the context of tumorigenesis, the moment when a persistent mutation occurs is also known to be of essence~\cite{koh2021mutational}. To better understand this aspect, in our future work we aim to explore the impact of the timing of a persistent mutation onset in our model. A simple way to approach this is to allow a time window for intact replication before applying the mutation to the function. One can then observe the differences produced in the prisoner set by introducing the mutation at different times (see Figure~\ref{mutation_timing}), as well as the evolution of the finite-time prisoners, and the appearance of ``clones'' along the timeline, similarly to clonal evolution in a fish plot~\cite{miller2016visualizing} when tracking oncogenesis (see Figure~\ref{timeline}).

\begin{figure}[h!]
\begin{center}
\includegraphics[width=0.75\textwidth]{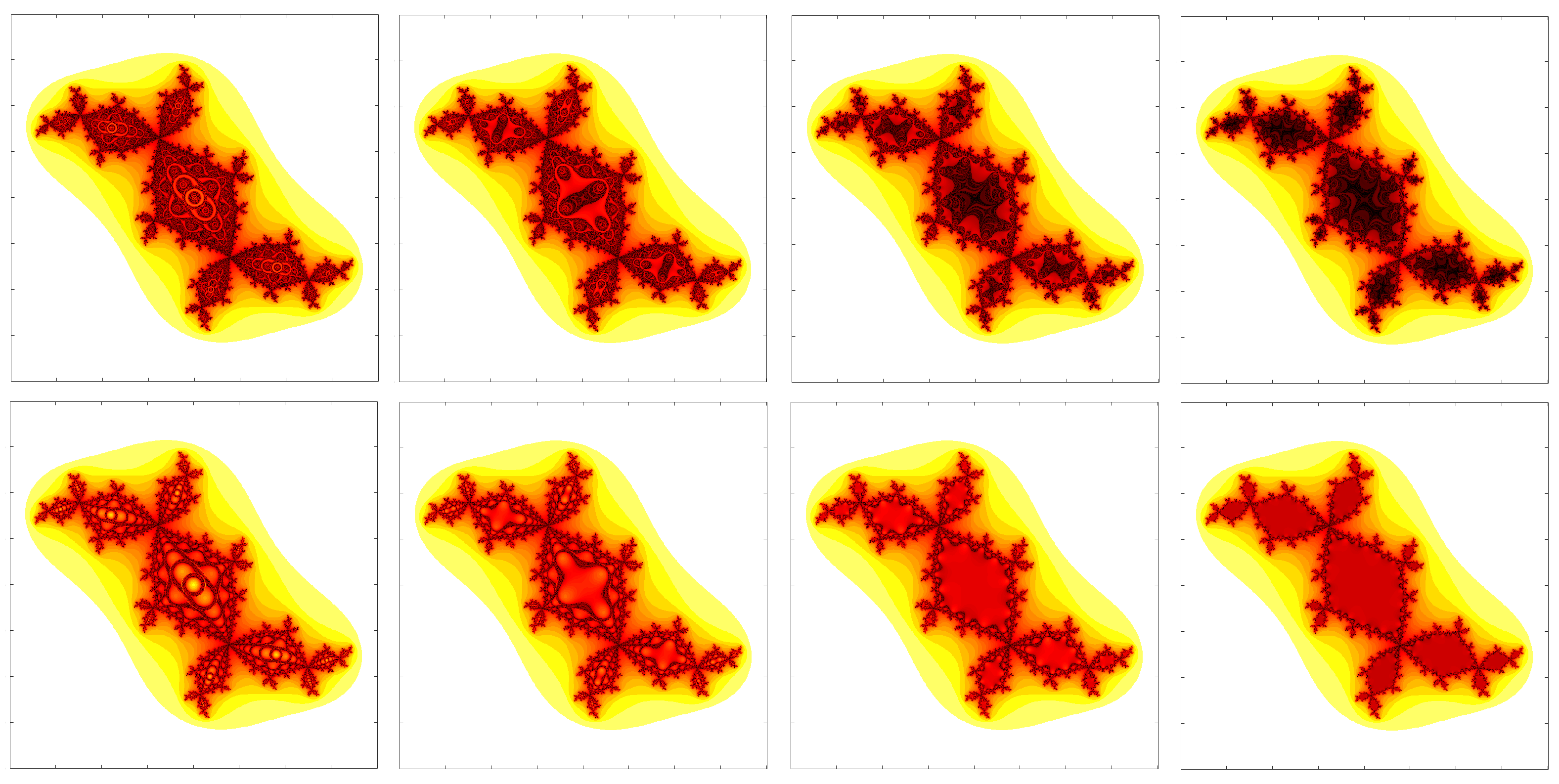}
\end{center}
\vspace{-4mm}
\caption{\small \emph{{\bf Changes induced in the prisoner sets by changing the timing of the mutation.} point-wise mutations $c_0=0$ (unit disk prisoner set, top row) and $c_0=-0.8-0.156i$ (dragon prisoner set, bottom row) are applied at the origin, with transition radius $R=0.1$, to the intact map $c_1 =-0.13+0.77i$ (Douady rabbit), after $n=3$, $n=5$, $n=10$ and respectively $n=15$ iterations steps (so that the replication is left intact for the specified number of initial iterations, after which the persistent mutation is applied).}}
\label{mutation_timing}
\end{figure}

\begin{figure}[h!]
\begin{center}
\includegraphics[width=0.95\textwidth]{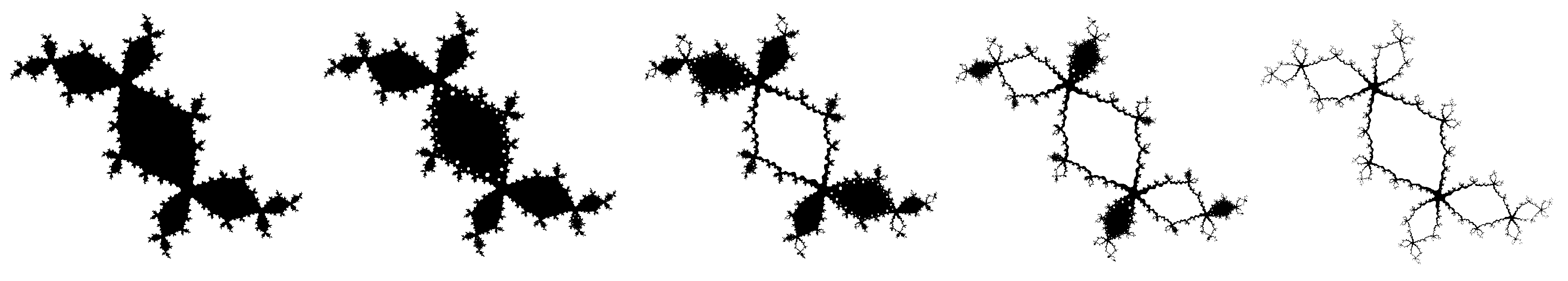}
\end{center}
\vspace{-4mm}
\caption{\small \emph{{\bf Changes in the finitely iterated prisoner set along the iteration timeline}, before and after mutation. The point-wise mutation $c_0=-0.8-0.156i$ is introduced at the origin at iteration step $n=15$, with fixed transition radius $R=0.3$. The panels show the points that are still within the escape disc after $n=20$, $n=21$, $n=22$, $n=23$ and respectively $n=24$ iterations.}}
\label{timeline}
\end{figure}

Finally, many of the choices made for the functions and transitions in our model are preliminary, and can be further altered to better match both biological and mathematical needs. For example, the interpolation used to define the decaying effects of mutation within the transition annulus is preliminary, and one of our goals is to improve its smoothness in future versions of the model. Analyticity of the function $f$ would bring significant mathematical advantages to the model than mere continuity, and would also likely represent a more realistic implementation of the biological phenomenology. More generally, the choice for using quadratic complex functions was based on their relative simplicity versus other models, and on the considerable amount of existing work and prior knowledge of dynamics in this family, both in the context of single iterated maps as well as that of random iterations. The same concepts can be investigated for a different dynamic family, creating the opportunity to explore the universality of mutation effects between families.

\section*{Acknowledgments}

We would like to thank Dr.~Luis Zapata Ortiz and Dr.~John Lowengrub for the useful conversations. This work is supported by a Research, Scholarship and Creative Activities Academic Year Undergraduate Research Experience (Longbotham), by a Simons Foundation Collaborative Grant for Mathematicians and an AMS-Simons PUI Faculty Grant (R\v{a}dulescu). \\


\bigskip
%
%
%
%
%
\clearpage

\section*{Appendix: Effects of mutations centered at $\xi^* \neq 0$}

\vspace{-4mm}
\begin{figure}[h!]
\begin{center} 
\includegraphics[width=0.7\textwidth]{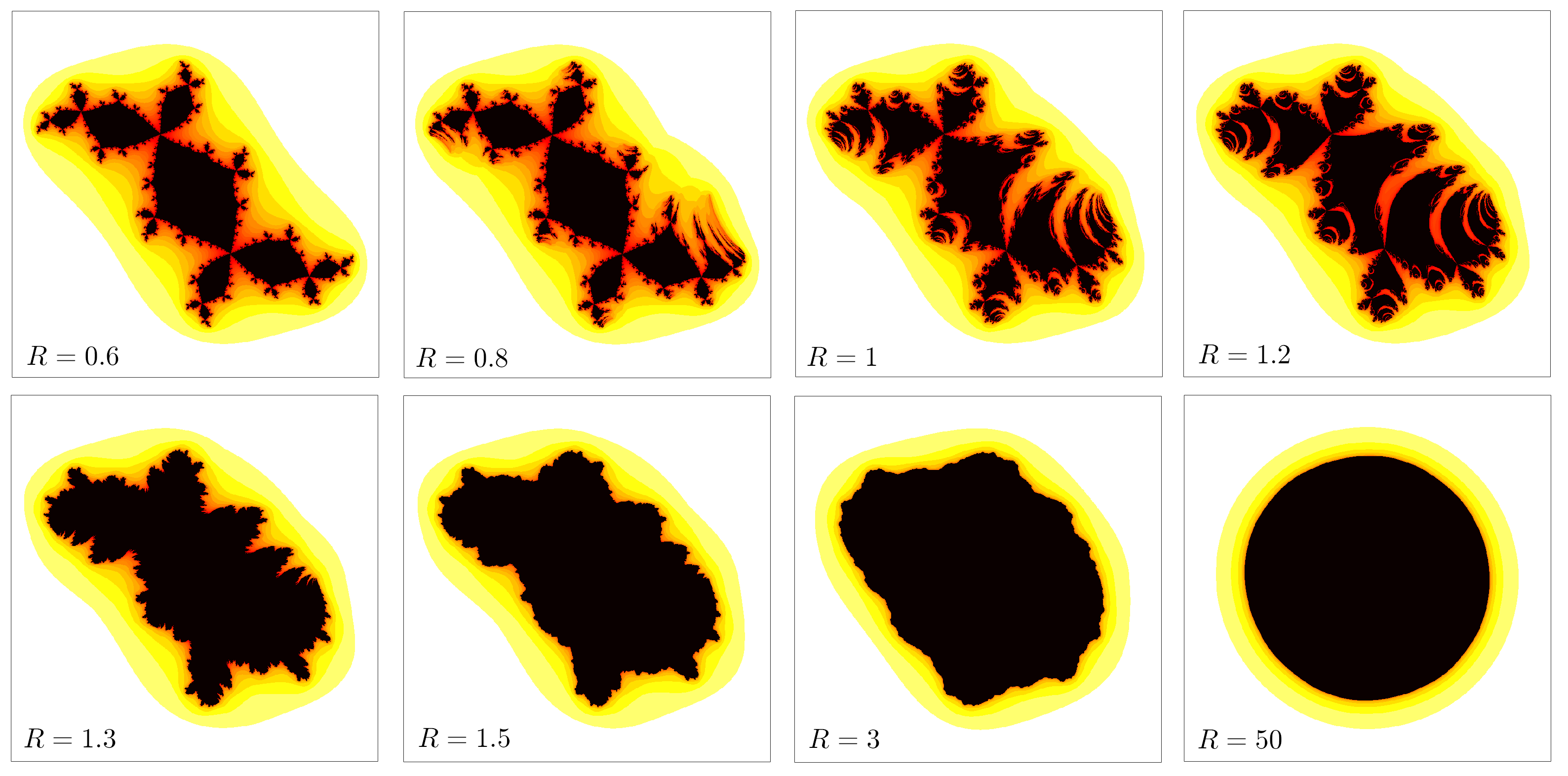}
\end{center}
\vspace{-4mm}
\caption{\small \emph{{\bf Dependence of the prisoner set on the mutation radius $r$.} All panels represent the mutation $c_0=0$ acting point-wise on $c_1= -0.13+0.77i$ at $\xi^*=1$. The transition radius is increased in panels (I) to (VIII) as follows: $R=0.6$; $R=0.8$; $R=1$; $R=1.2$; $R=1.3$; $R=1.5$; $R=3$; $R=50$.}}
\label{center_1}
\end{figure}

\vspace{-4mm}
\begin{figure}[h!]
\begin{center}
\includegraphics[width=0.7\textwidth]{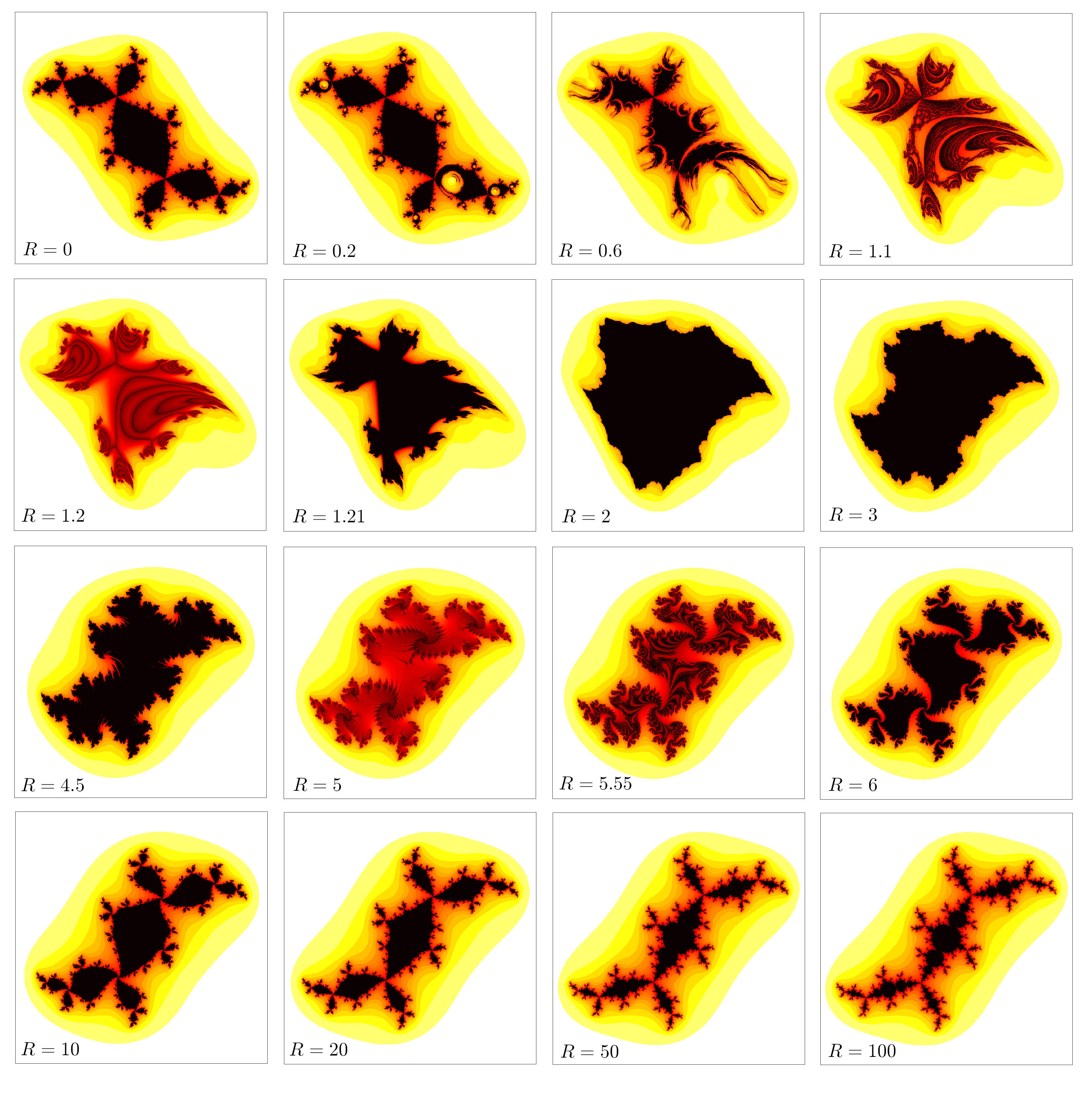}
\end{center}
\vspace{-5mm}
\caption{\small \emph{{\bf Dependence of the prisoner set on the mutation radius $r$.} All panels represent the mutation $c_0=-0.117-0.856i$ acting point-wise on $c_1= -0.13+0.77i$ at the focus $\xi^* = \frac{1+i}{2}$. The transition radius is increased in panels (I) to (XVI) as follows: $R=0$; $R=0.2$; $R=0.6$; $R=1.1$; $R=1.2$; $R=1.21$; $R=2$; $R=3$; $R=4.5$; $R=5$; $R=5.55$; $R=6$; $R=10$; $R=20$; $R=50$; $R=100$.}}
\label{center_diag}
\end{figure}

\clearpage
\bibliographystyle{plain}
\bibliography{references}

\end{document}